\documentclass{amsart}
\usepackage{amsmath,amssymb,amsthm,amscd, amsxtra, amsfonts, color}

\newtheorem{theorem}{Theorem}[section]
\newtheorem{corollary}{Corollary}[section]
\newtheorem{lemma}{Lemma}[section]

\newtheorem{remark}{Remark}[section]
\newtheorem{example}{Example}[section]

\newtheorem{proposition}{Proposition}[section]

\def \a{\alpha }

\def \l{\lambda }

\begin{document}

\newcommand{\wta}{{\rm {wt} }  a }
\newcommand{\R}{\frak R}
\newcommand{\cV}{\mathcal V}
\newcommand{\cA}{\mathcal A}
\newcommand{\cL}{\mathcal L}
\newcommand{\J}{\mathcal H}
\newcommand{\G}{\mathcal G}
\newcommand{\wtb}{{\rm {wt} }  b }
\newcommand{\bea}{\begin{eqnarray}}
\newcommand{\eea}{\end{eqnarray}}
\newcommand{\be}{\begin {equation}}
\newcommand{\ee}{\end{equation}}
\newcommand{\g}{\frak g}
\newcommand{\hg}{\hat {\frak g} }
\newcommand{\hn}{\hat {\frak n} }
\newcommand{\h}{\frak h}
\newcommand{\U}{\mathcal U}
\newcommand{\hh}{\hat {\frak h} }
\newcommand{\n}{\frak n}
\newcommand{\Z}{\Bbb Z}
\newcommand{\N}{{\Bbb Z} _{> 0} }
\newcommand{\Zp} {\Z _ {\ge 0} }
\newcommand{\C}{\Bbb C}
\newcommand{\Q}{\Bbb Q}
\newcommand{\1}{\bf 1}
\newcommand{\la}{\langle}
\newcommand{\ra}{\rangle}
\newcommand{\NS}{\bf{ns} }

\newcommand{\hf}{\mbox{$\tfrac{1}{2}$}}
\newcommand{\thf}{\mbox{$\tfrac{3}{2}$}}

\newcommand{\W}{\mathcal{W}}
\newcommand{\non}{\nonumber}
\def \l {\lambda}
\baselineskip=14pt
\newenvironment{demo}[1]%
{\vskip-\lastskip\medskip
  \noindent
  {\em #1.}\enspace
  }%
{\qed\par\medskip
  }

\def \l {\lambda}
\def \a {\alpha}

\keywords{vertex superalgebras, affine Lie algebras, Clifford
algebras, Weyl algebra, lattice vertex operator algebras, critical level, Z-algebras}
\title[On principal realization of $A_1 ^{(1)}$--modules at the critical level]{On principal realization of modules for the  affine Lie algebra $A_1 ^{(1)}$ at the critical level}

\author{ Dra\v zen Adamovi\' c}

\date{}
\address{Department of Mathematics, University of Zagreb,
Bijeni\v cka 30, 10 000 Zagreb, Croatia} \email {adamovic@math.hr}

\author{Naihuan Jing }
\address{Department of Mathematics, North Carolina State University, Raleigh, NC 27695, USA }
\email{jing@math.ncsu.edu}
\author{Kailash C. Misra}
\address{Department of Mathematics, North Carolina State University, Raleigh, NC 27695, USA}
\email{misra@math.ncsu.edu}

\markboth{} { }
\bibliographystyle{amsalpha}
\pagestyle{myheadings} \maketitle

\def \l {\lambda}
\def \a {\alpha}

\begin{abstract}
We present complete realization of irreducible $A_1 ^{(1)}$--modules at the critical level in the principal gradation. Our construction uses vertex algebraic techniques, the theory of twisted modules and representations of Lie conformal superalgebras. We also provide an alternative Z-algebra approach to this construction. All irreducible highest weight $A_1 ^{(1)}$--modules at the critical level are realized on the vector space $M_{\tfrac{1}{2} + \Z} (1) ^{\otimes 2}$ where $M_{\tfrac{1}{2} + \Z} (1)  $  is the polynomial ring ${\C}[\alpha(-1/2), \alpha(-3/2), ...]$. Explicit combinatorial bases for these modules are also given.

\end{abstract}

\maketitle

\section{Introduction}

Explicit realizations of affine Lie algebras and their representations have led to many important connections with other areas of mathematics and physics. The first explicit construction of the affine Lie algebra $A_1^{(1)}$ in terms of certain differential operators acting on a bosonic Fock space was given in \cite{LepWil78}. This was followed by a flurry of activities giving explicit realizations of affine Lie algebras and their integrable representations in both the homogeneous and principal gradation. Motivated by connections of partition identities to affine Lie algebras, Lepowsky and Wilson \cite{LepWil81, LepWil84} introduced certain nonassociative algebra called $Z$-algebra generated
by certain operators centralizing the action of a suitable Heisenberg subalgebra on the representation subspace. The $Z$-algebra has been used by several researchers to prove new and known combinatorial identities using affine Lie algebra representations in category $\mathcal{O}$.

The affine Lie algebra representations at the critical level are not in category $\mathcal{O}$, but seem to have
richer structures (cf. \cite{FG, AF}). In 1986 Wakimoto \cite{W-mod} gave a free field realization of the affine Lie algebra $A_1^{(1)}$ at an arbitrary level which has been extended to other affine Lie algebras by Feigin and Frenkel \cite{FeigFrenk, FeigFrenk90b} . These are called Wakimoto modules and they are reducible at integral levels. In \cite{DJM}, a realization of $A_1^{(1)}$ at the critical level is given using a Clifford type algebra. In \cite{A-2007} one of the authors introduced an infinite dimensional Lie superalgebra $\mathcal{A}$ and used vertex-algebraic techniques to construct a family of irreducible representations of $A_1^{(1)}$ at the critical level in the homogeneous picture. The principal realization of $A_1 ^{(1) }$--modules at the non-critical level are given in \cite{HJKOS} and \cite{HJM}. In these papers the authors study Wakimoto modules, and their relations with the deformed Virasoro algebra and $Z$-algebras.

In this paper, we use the vertex algebraic approach \cite{A-2007} to construct a family of irreducible $A_1 ^{(1)}$--modules at the critical level in the principal picture. Let $V_{-2} (sl_2)$ be the universal affine vertex algebra of level $-2$ associated to $\widehat{sl_2}$. It was proved \cite{A-2007} that $V_{-2} (sl_2)$ can be embedded into the vertex superalgebra ${\cV} \otimes F_{-1}$, where ${\cV}$ is the vertex superalgebra associated to the infinite--dimensional Lie superalgebra ${\cA}$ and
$F_{-1}$ is a lattice vertex superalgebra. For every irreducible, restricted ${\cA}$--module $U$ a family of irreducible $A_1 ^{(1)}$--modules $\mathcal{L}_s (U)$ were constructed in \cite{A-2007}.

Let $\Theta $ be the automorphism of $V_{-2} (sl_2)$ lifted from the $sl_2$--automorphism $\{ e \mapsto f, \ f \mapsto e, \ h \mapsto -h\}$.
We introduce the infinite--dimensional Lie superalgebra ${\cA} ^{tw}$  with basis
$$\mathcal{S}(n),  \ \mathcal{T}(n+1/2), \ \mathcal{G}(r), \ C , \qquad n \in {\Z}, \ r \in \tfrac{1}{2}{\Z} $$
and anti-commutation relations:
$$ \{ \mathcal{G}(r), \mathcal{G}(s) \} =  (-1) ^{2 r +1} ( 2 \delta_{r+s} ^{\Z}   \mathcal{S}(r+s) - \delta_{r+s} ^{ \tfrac{1}{2} + {\Z} }  (r-s) \mathcal{T}_{r+s} +
\tfrac{C}{3} \delta_{r+s} ^{\Z} ( r  ^ 2 - \tfrac{1}{4} ) \ \delta_{r+s,0} ), $$
$$ \mathcal{S}(n), \mathcal{T}(n+1/2), C \quad \mbox{in the center},$$
where $\delta_m ^X = 1$ if $ m \in X$, $\delta_m ^X = 0$ otherwise.
Then we show that the category of  twisted ${\cV}$--modules is analogous to the category of restricted ${\cA} ^{tw}$--modules with central charge $C=-3$.
One of our main results is the following.

%
%
%
\begin{theorem}
Assume that $U ^{tw}$ is an irreducible, restricted ${\cA} ^{tw}$--module, and $F_{-1} ^{T}$ a twisted $F_{-1}$--module. Then $U ^{tw} \otimes F_{-1} ^{T}$ has the structure of an  irreducible $\widehat{sl_2}[\Theta]$--module at the critical level.
\end{theorem}
The proof of this result  will be presented in Section  \ref{sl2-section} (cf. Theorem \ref{main-t1}).

As in \cite{A-2007} and \cite{A-2013} we prove the irreducibility of a larger family of $\Theta$--twisted $V_{-2}(sl_2)$--modules. For every $\chi \in {\C}( z^{\tfrac{1}{2} })$ the  twisted $F_{-1}$-modules  $F^{T_i}(\chi)$, ($i=1,2$)  have the structure of restricted ${\cA}^{tw}$--module which is uniquely determined by certain twisted field.
In Section \ref{constr-2} we construct a family of irreducible modules for ${\cA}^{tw}$.
In fact  we prove the following classification result.

\begin{theorem} \label{ired-uvod}
Assume that $p \in \tfrac{1}{2}{\Zp}$  and that
$$\chi(z) = \sum_{ k=-2p} ^{\infty} \chi_{ -\tfrac{k}{2} } z ^{ \tfrac{k}{2} -1}. $$
Then $F^{T_i}(\chi)$ is irreducible ${\cA}^{tw}$--module if and only if one of the following conditions hold:
\bea
\label{cond-i1} && p >0  \quad \mbox{and} \  \chi_p \ne 0, \\
&& p = 0 \quad \mbox{and} \ \chi_0 \in ({\C} \setminus \tfrac{1}{2}{\Z} ) \cup \{ \frac{1}{2} \}, \label{cond-i2} \\
&& p=0 \quad \mbox{and} \ \chi_0 - \frac{1}{2}  = \ell \in  \tfrac{1}{2}{\Zp}  \quad \mbox{and}   \ \det (A (\chi)) \ne 0  \label{cond-i3}  \eea
where
$$ A(\chi) = \left(
\begin{array}
[c]{ccccc}%
2 S(-1)  & 2 S(-2) & \cdots & \  & 2 S(-2 \ell)\\
\ell ^2 - (\ell-1) ^2 & 2 S(-1)  & 2 S(-2) & \cdots & 2 S(-2 \ell -1) \\
0 & \ell ^2 - (\ell -2) ^2 &  2 S(-1) & \cdots & 2 S(-2\ell-2) \\
\vdots & \ddots & \ddots & \ddots & \vdots\\
0 & \cdots & 0 & \ell ^2 - (\ell  -2 \ell +1) ^2 & 2 S(-1)
\end{array} \right)
$$
and
$$S(z) = \frac{1}{2} ( \chi^ {(1)}(z) ) ^2 + \partial_z \chi^ {(1)}(z) ) =\sum_{n \in \Z} S(n) z ^ {-n-2}  $$
(here $\chi^{(1)} $  is the integral part of $\chi$).
\end{theorem}
The proof of this theorem will be presented in Sections  \ref{constr-2} and \ref{korepondencija}. 

As in \cite{A-2013}, we prove that the condition (\ref{cond-i3}) can be replaced by the condition
$$S_{2 \ell} (-2 \chi ^ {(1)}  _{-1} , - 2 \chi  ^ {(1)}  _{-2} ,  \cdots ) \ne 0, $$  where
$S_{2 \ell}$ is a Schur polynomial.

As an application, all highest weight irreducible
$\widehat{sl}_2$-modules at the critical level are constructed in the principal gradation.
These modules are realized on the vector space
$$M_{\tfrac{1}{2} + \Z} (1) ^{\otimes 2} $$
where $M_{\tfrac{1}{2} + \Z} (1)  $  is a polynomial ring ${\C}[\alpha(-1/2), \alpha (-3/2), ...]$. It is worthwhile to mention  that this result is in agreement with the Kac-Kazhdan character formula for highest weight modules at the critical level (cf. \cite{KK}, \cite{efren}, \cite{W-mod}).

In fact, for $t \in \{-1 \} \cup {\C} \setminus {\Z}$, the Kac-Kazhdan
character formula for the irreducible $\widehat{sl_2}$--module $L(\mu)$, $\mu = -(2+ t) \Lambda_0 + t \Lambda_1$
says that
$$ \mbox{ch}_{L(\mu)} = e ^{\mu} \prod_{ n=1 } ^{\infty} (1 -e ^{ \alpha - n \delta} )^{-1} (1 - e ^{-\alpha - (n-1) \delta} ) ^{-1}. $$
By taking  $q =e ^{-\delta}$ and $q ^{{1/2}} = e ^{-\alpha}$ we get

$$ \mbox{ch}_{L(\mu)} = e ^{\mu} \prod_{ n=1 } ^{\infty} (1 -q ^{n-1/2} )^{-2},$$
which also indicates that $L(\mu)$ admits certain realization on $M_{\tfrac{1}{2} + \Z} (1) ^{\otimes 2} $.

We compare our realization with the construction presented in \cite{DJM}. We prove that at the critical level the vacuum space $\Omega(V_{-2}(sl_2) )$ of the universal affine vertex algebra $V_{-2}(sl_2)$ is isomorphic to a quotient $\widetilde{\cV}$ of the vertex algebra ${\cV}$. Moreover, the vacuum space of the simple affine vertex algebra $L_{-2}(sl_2)$ is a simple quotient $\overline{F}$ of ${\cV}$ introduced in \cite{A-2007}. In particular, our result shows that  the irreducible $A_1^{(1)}$--modules at the principal gradation have the form of a tensor product $U ^{tw} \otimes M_{ \tfrac{1}{2} + \Z}(1)$ where $U ^{tw}$ is an irreducible twisted  module for $\Omega(V_{-2}(sl_2))$ and $M_{ \tfrac{1}{2} + \Z} (1)$ is an irreducible twisted module for the Heisenberg vertex algebra $M_{\Z} (1)$.

Our approach can be used for describing bases of  modules at the critical level. For $\mu = (-2 -t ) \Lambda_0 + t \Lambda_1$ we also use $L(\mu)$ to denote the
$\Theta$--twisted $V_{-2}(sl_2)$--module with vertex operator map $Y^{tw}$.
Define $$ G(z) = E^{-} _{tw} (-\frac{h}{2},z) Y^{tw} (e(-1){\bf 1}, z) E^{+}_{tw} (- \frac{h}{2}, z) = \sum_{ n \in \tfrac{1}{2}{\Z} } G(n) z ^{-n-1}. $$
We prove the following result which gives the basis for $L(\mu)$ in the principal picture.
\begin{theorem} \label{th1.1-intr}  Assume that $\mu = (-2 -t ) \Lambda_0 + t \Lambda_1$.
\item[(1)] If  $t \in \{-1 \} \cup ({\C} \setminus {\Z} )$,  then the set of vectors
$$  G(-n_1) G(-n_2) \cdots G(-n_r) h(-m_1) \cdots h(-m_s) v_{\mu}  $$
such that
$ r, s \ge 0$, $n_i, m_j \in \tfrac{1}{2}{\Zp}$ and
$$ n_1 > n_2 > \cdots >n_r > 0 , \quad m_1 \ge m_2 \ge \cdots \ge  m_s > 0$$
is a basis of $L(\mu)$ in the principal picture.

\item[(2)] If $t \in {\Zp}$, then the set of vectors
$$  G(-n_1) G(-n_2) \cdots G(-n_r) h(-m_1) \cdots h(-m_s) v_{\mu}  $$
such that
$ r, s \ge 0$, $n_i, m_j \in \tfrac{1}{2}{\Zp}$, $n_i \ne \tfrac{t}{2} + \tfrac{1}{2}$,  and
$$ n_1 > n_2 > \cdots >n_r > 0 , \quad m_1 \ge m_2 \ge \cdots \ge  m_s > 0$$
is a basis of $L(\mu)$ in the principal picture.
\end{theorem}

Similar basis can be constructed for all irreducible modules in Theorem \ref{ired-uvod}.

 \section{Certain vertex algebras and their twisted modules}
\label{certain}
In this section
we review some results on lattice, fermionic and bosonic vertex algebras and their twisted modules
(cf. \cite{FLM}, \cite{FHL}, \cite{K}, \cite{LL}).

\subsection{ Clifford vertex superalgebras and their twisted modules}
The Clifford algebra $CL$ is a complex associative algebra generated by
$$ \Psi^{\pm}(r) , \  r  \in \hf + {\Z},$$ subject to the relations
\bea
&& \{\Psi^{\pm}(r) , \Psi^{\mp}(s) \} = \delta_{r+s,0}, \quad
 \{\Psi^{\pm}(r) , \Psi^{\pm}(s)\}=0,
\nonumber
\eea
where $r, s \in \frac{1}{2} + {\Z}$. The Clifford
algebra $CL ^{tw} $ is a complex associative algebra generated by
$$ \Phi (r) , \  r  \in \hf  {\Z},$$ under the relations
\bea
&& \{\Phi(r) , \Phi (s) \} =  - (-1) ^{ 2 r} \delta_{r+s,0}; \quad \nonumber
\eea
where $r, s \in \frac{1}{2} {\Z}$.

Let $F$ be the irreducible $CL$--module generated by
 the
cyclic vector ${\1}$ such that
$$ \Psi^{\pm} (r) {\1} = 0 \quad
\mbox{for} \ \ r > 0 .$$
A basis of $F$ is given by
$$ \Psi^{+ }({-n_1-{\hf}})  \cdots \Psi^{+}({-n_r-{\hf}})  \Psi^{-}({-k_1-{\hf}})  \cdots \Psi^{-}({-k_s-{\hf}})
 {\1} $$
where $n_i, k_i \in {\Zp}$,  $n_1 >n_2 >\cdots >n_r  $, $k_1 >k_2
>\cdots
>k_s $.

Define the following   fields on $F$
$$   \Psi^{+}(z) = \sum_{ n \in   {\Z}
 } \Psi^{+}(n+{\hf} )  z ^{-n- 1}, \quad  \Psi^{-} (z) = \sum_{ n \in {\Z}
 } \Psi^{-} (n+{\hf} )  z ^{-n-1}.$$

 The fields $\Psi^{+}(z)$ and $\Psi^{-}(z)$ give arise to the unique simple vertex superalgebra
 structure on $F$ \cite{K, FB}. As a vertex operator superalgebra, $F$ has the involution $\Theta_{F}$ induced
 from the automorphism of the Clifford algebra $CL$:
\begin{equation*} \Theta  \Psi^{\pm}(n+ {\hf}) =  \Psi^{\mp}(n+ {\hf}).
\end{equation*}
 Let $\phi^{(i)}$ be the following $\pm 1$-eigenvectors of $\Theta_F$:
 $$ \phi^{(1)} = \frac{1}{\sqrt{2}} (\Psi^+(-1/2) + \Psi^- (-1/2) ) {\bf 1}, \quad \phi^{(2)}  = \frac{1}{\sqrt{2}} (\Psi^+(-1/2) - \Psi^- (-1/2) ) {\bf 1}. $$

 Let $F ^{T_i}$ $(i=1, 2)$ be the irreducible $CL^{tw}$-module spanned by (as a vector space)
 \cite{AM-sigma, Xu}
 $$\bigwedge( \Phi (r) \ \vert \ r < 0) $$
 with the module structure given by
\begin{align*}
\Phi^{(1)}  (z) &= Y_{ F ^{T_i} } ( \phi ^{(1)}, z) =\sum_{ n \in {\Z} } \Phi (n+ 1/2) z ^{-n-1},\\
\Phi^{(2)}  (z) &= Y_{ F ^{T_i} } ( \phi ^{(2)}, z) =\sum_{ n \in \tfrac{1}{2} +{\Z} } \Phi (n+ 1/2) z ^{-n-1}
\end{align*}
such that $\Phi(z)=\Phi^{(1)}(z)+\Phi^{(2)}(z)$
and $\Phi(0) \equiv (-1) ^{i} \frac{1}{\sqrt{-2}} \ \mbox{Id} $. Then  $F^{T_1}, \ F^{T_2}$
are the only non-isomorphic irreducible $\Theta_{F}$--twisted modules of $F$.
\subsection{Commutative vertex algebra $M(0)$ and its twisted modules}

Let $M(0) = {\C}[{\gamma}^{+}(n), {\gamma}^{-}(n) \  \vert \ n <0
]$ be the commutative vertex algebra generated by the fields
$${\gamma} ^{\pm} (z) = \sum_{ n  < 0} {\gamma}^{\pm} (n) z ^{-n-1}. $$
 Let $\chi^{\pm}(z) = \sum_{n \in {\Z} }
\chi^{\pm}_{n} z^{-n-1} \in {\C} ((z))$, and let $M(0, \chi^{+},
\chi^{-})$ be the $1$--dimensional irreducible $M(0)$--module
such that every element ${\gamma}^{\pm}(n)$ acts on
$M(0, \chi^{+}, \chi^{-})$ as multiplication by $\chi^{\pm}_n \in
{\C}$.

The vertex algebra $M(0)$ has an automorphism $\Theta_{M(0)}$ of order two given by
$$ \Theta_{M(0)} (\gamma ^{\pm} (n) ) = \gamma ^{\mp} (n). $$

Let $$ \gamma^{(1)} = \frac{\gamma ^{+} + \gamma ^{-}}{2}, \ \gamma^{(2)} = \frac{\gamma ^{+} - \gamma ^{-}}{2}. $$
$\Theta_{M(0)}$--twisted $M(0)$--modules are restricted modules for the commutative Lie algebra spanned by
$$\gamma  (n), \quad n \in {\hf}{\Z}. $$

For every $\chi(z) \in {\C} (( z^{1/2}))$, let $M^{tw}(0, \chi)$ be $1$--dimensional, irreducible, $\Theta_{M(0)}$--twisted $M(0)$--module with the property that $\gamma(n)$ acts as multiplication by $\chi(n)$. We have
$$ Y_{ M^{tw}(0, \chi) } ( \gamma^{(1)}(-1) {\bf 1} , z) = \chi^{(1)} (z), \ Y_{ M^{tw}(0, \chi) } ( \gamma^{(2)}(-1) {\bf 1} , z) = \chi^{(2)} (z), $$
where
$$ \chi(z) = \chi^{(1)}  (z) + \chi^{(2)} (z), $$
$$  \chi^{(1)}  (z) = \sum_{ n \in {\Z} } \chi (n+ 1/2) z ^{-n-1}, \quad   \chi^{(2)}  (z) = \sum_{ n \in \tfrac{1}{2} +{\Z} } \chi (n+ 1/2) z ^{-n-1}.$$

\subsection{The vertex superalgebra $F_{-1}$ and its twisted modules }

Let $F_{-1}$ be the lattice vertex superalgebra associated to the lattice $L= {\Z} \beta$
with the inner product $\langle \beta, \beta \rangle = -1$.
Let $\h = {\C} \otimes _{\Z} L$, the Heisenberg and twisted Heisenberg algebras
$${\hh}_{\Z} = {\h} \otimes {\C}[t,t ^{-1}] \oplus {\C} C , \qquad
{\hh}_{ \tfrac{1}{2} + \Z} = {\h} \otimes t ^{1/2} {\C}[t,t ^{-1}] \oplus {\C} C $$ are defined as usual
with the bracket: for $\alpha, \gamma\in \h$
\begin{equation}
[\alpha(m), \gamma(n)]=m\delta_{m, -n}\langle \alpha, \gamma\rangle C, \quad m, n\in \mathbb Z\ \mathrm{or}
\ \mathbb Z+1/2.
\end{equation}

Let $M_{\Z} (1) = S (\h \otimes t ^{-1} {\C}[t^{-1}]) $ and  $M_{\tfrac{1}{2} + \Z} (1) = S (\h \otimes t ^{-1/2} {\C}[t^{-1}]) $
be the level one irreducible modules for ${\hh}_{\Z} $  and ${\hh}_{ \tfrac{1}{2} + \Z}$ respectively.

The vertex superalgebra $F_{-1}$ can be realized as  $M_{\Z} (1) \otimes {\C}[L]$, where ${\C}[L]$ is the group algebra with basis $\{ e ^{\alpha}, \ \alpha \in L\}$.  There is a linear map
$$ Y : F_{-1}  \rightarrow \mbox{End} (F_{-1} ) [[z,z ^{-1}]], \quad v \mapsto Y(v,z) = \sum_{n \in {\Z} } v_n z ^{-n-1} $$
such that $(F_{-1}, Y, {\bf 1}) $ is a vertex superalgebra.
For $\alpha \in L$ we have
$$ Y(e ^{\alpha}, z ) = E ^{-} (-\alpha, z) E ^{+} (-\alpha, z) e ^{\alpha} z ^{\alpha(0)} $$
where
$$ E^{\pm}  (-\alpha, z) = \exp \left( \sum_{k=1} ^{\infty} \frac{\alpha(\pm k)}{\mp k}  z ^{\mp k} \right). $$

As a vertex algebra, $F_{-1}$ is generated by $e ^{\beta}$ and $e ^{-\beta}$. We have the following $\Z$--gradation:
$$ F_{-1} = \bigoplus_{i \in {\Z}} F_{-1} ^{(i)}, \quad  F_{-1} ^{(i)}= M_{\Z} (1). e^{i \beta}. $$

The vertex algebra $F_{-1}$ has an automorphism $\Theta_{F_{-1}}$ lifted from the automorphism $\beta \mapsto -\beta$ of the lattice $L$. Then
$$ \Theta_{F_{-1}} (e ^{\pm \beta} ) = e ^{\mp \beta}, \Theta_{F_{-1}} ( \beta(-1) ) = - ( \beta(-1) ). $$
There are two non-isomorphic irreducible $\Theta_{F_{-1}}$-twisted $F_{-1}$-modules $F_{-1} ^{T_i}$, $i=1,2$ which are realized on the irreducible ${\hh}_{{\hf} + {\Z} }$--module
$M_{\hf + {\Z}} (1)$ (for details see \cite{FLM}).
The module structure is defined by twisted vertex operators:
$$ Y_{tw} : F_{-1}  \rightarrow \mbox{End} (F_{-1} ^{T_i} ) [[z^{1/2},z ^{-1/2}]], \quad v \mapsto Y_{tw} (v,z) = \sum_{n \in \tfrac{1}{2}  {\Z} } v_n z ^{-n-1}. $$
Here, for $\alpha \in L$ we have
\bea \label{for-vo1}  && Y_{tw}(e ^{\alpha}, z ) = -(2 \sqrt{-1})  ^{- \langle \alpha, \alpha \rangle  }E_{tw} ^{-} (-\alpha, z) E_{tw} ^{+} (-\alpha, z)  (-1) ^i  \eea
where
\bea \label{for-vo2} &&E_{tw}^{\pm}  (-\alpha, z) = \exp \left( \sum_{k \in {\Zp} + \tfrac{1}{2} }  \frac{\alpha(\pm k)}{\mp k}  z ^{\mp k} \right). \eea

\begin{remark} By the boson--fermion correspondence, the fermionic vertex algebra $F$ can be realized as a lattice vertex algebra $V_{\Z}$ and the twisted modules
$F^{T_i}$ can be also realized on the irreducible ${\hh}_{{\hf} + {\Z} }$--module
$M_{\hf + {\Z}} (1)$.
\end{remark}

\section{ The vertex superalgebra ${\cV}$ and its twisted  modules}
\label{ver-def}

Recall first the construction of the vertex superalgebra ${\cV}$ from \cite{A-2007}.
Let ${\mathcal F}$ be the vertex superalgebra generated by the
fields $\Psi^{\pm} (z)$ and ${\gamma}^{\pm}(z)$. Therefore
${\mathcal F} = F \otimes M(0)$. As in \cite{A-2007}, denote by
${\cV}$ the vertex subalgebra of the  vertex superalgebra
${\mathcal F}$ generated by the following vectors
\bea
 \tau^{\pm} &=& (\Psi^{\pm} (-\tfrac{3}{2}) + {\gamma}^{\pm} (-1)
 \Psi^{\pm}(-\hf)) {\1}, \label{def-tau} \\
 j &=& \frac{ {\gamma}^{+} (-1) - {\gamma}^{-}(-1)}{2} {\1}, \label{def-j} \\
 \nu &=& \frac{ 2 {\gamma}^{+} (-1) {\gamma}^{-}(-1) + {\gamma}^{+}(-2) + {\gamma}^{-}(-2)}{4}
  {\1} .
\label{def-nu}  \eea
The  vertex superalgebra structure on ${\cV}$     is given by
the following fields \bea && G^{\pm} (z) = Y(\tau ^{\pm} ,z)
= \sum _{n \in {\Z} } G ^{\pm} (n+ {\hf} ) z ^{-n-2}, \label{polje-g} \\
 && S (z) = Y(\nu,z)
= \sum _{n \in {\Z} } S({n })  z ^{-n-2}, \label{polje-s} \\
&& T (z) = Y(j,z) = \sum _{n \in {\Z} } T(n )  z ^{-n-1}.
\label{polje-t} \eea

It turns out that the components of the field operators generate a Lie superalgebra.
By abusing notation, let ${\cA}$ be the Lie superalgebra with basis $ S(n), T(n),
{G} ^{\pm} (r), C$, $n\in {\Z}$, $ r\in {\hf} + {\Z}$ with the
(anti)commutation relations given by \bea && [S(m),S(n)] = [S(m),
T(n) ] = [S(m), G^{\pm} (r) ] =0,
\non \\
&& [T(m), T(n)]= [T(m), G^{\pm} (r)] =0, \nonumber \\ && [C,
S(m)]= [C, T(n)] = [C,
G^{\pm}(r)]=0, \nonumber \\
   && \{ G ^{+} (r),
G ^{-} (s) \} = 2 S({r+s}) +
(r-s) T({r+s}) + \tfrac{C}{3} ( r  ^ 2 - \tfrac{1}{4} ) \delta_{r+s,0}, \non \\
 && \{ G ^{+} (r), G ^{+} (s) \}= \{ G ^{-} (r), G ^{-} (s) \} = 0 \non
\eea for all $n \in {\Z}$, $r,s \in {\hf} + {\Z}$.

Using the commutator formulae for vertex superalgebras, we have
that the components of fields (\ref{polje-g})-(\ref{polje-t})
satisfy the (anti)commutation relation for the Lie superalgebra
${\cA}$ with the central element $C$ acting as multiplication
by $C=-3$.
Recall from \cite{A-2007} that ${\cV}$ admits a natural $\Z$--gradation:
\bea &&  {\cV} = \bigoplus_{i \in {\Z} } {\cV} ^{(i)}. \label{grad-1} \eea

Let $M_T(0)$ be the vertex subalgebra of ${\cV}$ generated by the field $T(z)$ and let $\widetilde{I} = U(\cA). M_T(0)$ be the ideal in $\mathcal{V}$ generated by $M_T(0)$.  Set
$$ \widetilde{\mathcal{V} } = \mathcal{V} / \widetilde{I}. $$

As in  \cite{A-2007}, let ${\mathcal V}^{com}$ be the subalgebra of ${\cV}$  generated by $T(z)$ and $S(z)$ and let  $I ^{com}= U(\mathcal A)$ be the ideal in  ${\mathcal V}$ generated by ${\mathcal V} ^{com}$.
Let $\overline{F} = {\cV} / I ^{com}$.  Then $\overline{F}$ is a simple vertex superalgebra, which can  be realized as a subalgebra of Clifford vertex superalgebra $F$ generated by $\partial_z \Psi^+ (z)$ and $\partial_z \Psi^- (z)$.

Since ideals $\widetilde{I}$ and  $I ^{com} $ are  $\Z$--graded, the gradation (\ref{grad-1}) induces  the $\Z$--gradation on $\widetilde{\mathcal{V} }$  and $\overline{F}$:
$$   \widetilde{\cV} = \bigoplus_{i \in {\Z} } \widetilde{\cV} ^{(i)}, \quad  \overline{F} = \bigoplus_{i \in {\Z} } \overline{F} ^{(i)}.  $$

We have  the generators of $\widetilde{\cV}$
\bea \label{tilde-gen} &&\widetilde{\tau}^{\pm} := \tau^{\pm} + \widetilde{I},   \ \widetilde{\nu}= j + \widetilde{\nu}. \eea
Let $${\Theta}_1  = {\Theta}_F \otimes {\Theta}_{M(0)}. $$
It is clear that $\Theta_1$ is an automorphism of $\mathcal{F}$ which is $\cV$--invariant. Let $$ \Theta_{\cV}= {\Theta}_1 \vert \ {\cV}.$$
Clearly,
$$ {\Theta}_{\cV} (\tau ^{\pm}) = \tau ^{\mp}, \ {\Theta}_{\cV} (j) = -j, \ {\Theta}_{\cV} (\nu) = \nu.$$
Therefore $ {\Theta}_{\cV} $ is an  automorphism of order two of ${\cV}$. Since the ideal  $\widetilde I  $ is ${\Theta}_{\cV}$--invariant, it induces the automorphism of   $\widetilde { \cV} $ which we shall denote by 
${\Theta}_{\widetilde {\cV} }$.

Let

$$ \tau ^{ (1)} = \frac{1}{\sqrt{2}} ( \tau ^+ + \tau ^-) ,  \quad  \tau ^{ (2)} = \frac{1}{\sqrt{2}} ( \tau ^+ - \tau ^-). $$

Recall the infinite dimensional Lie suparalgebra ${\cA}^{tw}$ with basis
$$\mathcal{S}(n),  \ \mathcal{T}(n+1/2), \ \mathcal{G}(r), \ C , \qquad n \in {\Z}, \ r \in \tfrac{1}{2}{\Z} $$
and anti-commutation relations:
$$ \{ \mathcal{G}(r), \mathcal{G}(s) \} =  (-1) ^{2 r +1} ( 2 \delta_{r+s} ^{\Z}   \mathcal{S}(r+s) - \delta_{r+s} ^{ \tfrac{1}{2} + {\Z} }  (r-s) \mathcal{T}_{r+s} +
\tfrac{C}{3} \delta_{r+s} ^{\Z} ( r  ^ 2 - \tfrac{1}{4} ) \ \delta_{r+s,0} ), $$
$$ \mathcal{S}(n), \mathcal{T}(n+1/2), C \quad \mbox{in the center}, $$
with $\delta_m ^S = 1$ if $ m \in S$, $\delta_m ^S = 0$ otherwise.

\begin{remark}
Lie superalgebra ${\cA}^{tw}$ is similar to the twisted $N=2$ superconformal Lie superalgebra. The only
difference is that ${\cA} ^{tw}$ contains a large center.
\end{remark}

Assume that $(M^{tw}, Y^{tw})$ is any $\Theta_{\cV}$--twisted ${\cV}$--module. Define
\bea
G ^{(1)} (z) &=& Y^{tw} ( \tau^{(1)}, z  ) = \sum_{n \in {\Z} } G (n+ \frac{1}{2}) z ^{-n-2}, \nonumber \\
G ^{(2)} (z) &=& Y^{tw} ( \tau^{(2)}, z ) = \sum_{n \in \tfrac{1}{2} + {\Z} } G (n+ \frac{1}{2}) z ^{-n-2}, \nonumber \\
G(z)  & = & G^{(1)} (z) + G ^{(2)} (z) = \sqrt{2} Y^{tw} (\tau ^{+}, z), \nonumber \\
S(z) & = & Y^{tw} (\nu,z) = \sum_{n \in {\Z} } S(n) z ^{-n-2}, \nonumber \\
T(z) & = & Y^{tw} (j ,z) = \sum_{n \in \tfrac{1}{2} +{\Z} } T(n) z ^{-n-1}. \nonumber
\eea

By using commutator formulas for twisted modules one proves that the components of the fields $G(z), S(z), T(z)$ satisfy the commutation relations for the Lie superalgebra ${\cA} ^{tw}$
(see also  \cite{B}, \cite{FLM}, \cite{Li-twisted}, \cite{K}).
So we have:
\begin{theorem}
The category of $\Theta_{\cV}$--twisted ${\cV}$--modules coincides with the category of restricted modules for Lie superalgebra ${\cA} ^{tw}$. A ${\Theta}_{\cV}$--twisted module $U^{tw}$ is irreducible if and only if it is irreducible as an ${\cA} ^{tw}$--module.
\end{theorem}

 Now   we shall consider a family of
 twisted ${\cV}$--modules.
For ${\chi} \in {\C}((z^{1/2}))$ we set
$F^{T_i} ({\chi}) :=F^{T_i}  \otimes M ^{tw} (0,\chi)$.
Then $F ^{T_i} (\chi)$ is a $\Theta_{\cV}$--twisted  module for the vertex
superalgebra ${\cV}$.
Since $ M^{tw} (0,\chi)$ is one-dimensional, we have that
as  a vector space

\bea  \label{identification-prva} F ^{T_i} (\chi)\cong
F ^{T_i} \cong { \bigwedge} \left(\Phi (-n) \ \vert \ n \in \tfrac{1}{2} {\Z}, n > 0 \right).
\eea
We have:

 \bea
 \tau^{(1)} &=& \left(  \phi^{(1)} (-3/2) + \phi^{(1)} (-1/2) \gamma^{(1)} (-1) +   \phi^{ (2)} (-1/2) \gamma ^{(2)} (-1) \right) {\bf 1}, \nonumber \\
 \tau^{(2)} &=& \left(  \phi^{(2)} (-3/2) + \phi^{(1)} (-1/2) \gamma^{(2)} (-1) +   \phi^{ (2)} (-1/2) \gamma ^{(1)} (-1) \right) {\bf 1}. \nonumber
 \eea
Then the field
\bea  G(z) = \partial_{z} \Phi(z) + \chi(z) \Phi(z) \label{twisted-field-3} \eea
uniquely defines the (twisted) module structure on $F ^{T_i} (\chi)$.

\begin{remark}
Note that if $\chi^{(2)} (z) = 0$, then $F ^{T_i} (\chi)$ is a  $\Theta_{\widetilde \cV}$--twisted $\widetilde{\cV}$--module.
\end{remark}

\section{Construction of irreducible ${\cA}^{tw}$--modules}
\label{constr-2}
In this section we shall describe some irreducible  ${\cA}^{tw}$--modules. Construction of irreducible  ${\cA}$--modules was given in    \cite{A-2007} and \cite{A-2013}. Our results will provide a twisted generalization of these modules. We start with the following simple, but important example.

 \begin{example} \label{ex-1}
 Assume that $\chi(z) = \frac{\lambda}{z}$.  Then the action of ${\cA}^{tw}$ is given by
 $$ G(n+1/2) = (\lambda -n-1 ) \Phi(n+1/2) \quad (n \in \tfrac{1}{2} {\Z} ). $$
 \item[(1)]If   ${\lambda} \in ({\C} \setminus \tfrac{1}{2} {\Z}) \cup \{\frac{1}{2}\} $,
 using the fact that  $F ^{T_i}$ is irreducible as $CL^{tw}$--module, we get that
 $F ^{T_i} ( \frac{\lambda}{z})$ is an irreducible  ${\cA}^{tw}$--module and therefore irreducible $\Theta_{\cV}$--twisted $\cV$--module.
 \item[(2)] If ${\lambda} \in \tfrac{1}{2} {\Z} $, $ \lambda < \frac{1}{2}$, one easily sees that $F^{T_i} ( \lambda /z)$ has a proper irreducible submodule:
$$ \overline{F} ^{T_i} (\lambda / z) = \mbox{Ker}_{ \overline{F} ^{T_i}  }  \Phi (1/2- \lambda )   \cong \bigwedge
\left(   \Phi(n)  \ \vert \ n \in \tfrac{1}{2} {\Z}, n < 0,  n \ne  \lambda -1/2 \right). $$

\item[(3)] If   ${\lambda} \in \tfrac{1}{2} {\Z} $, $ \lambda > \frac{1}{2}$, then $F^{T_i} ( \lambda /z)$ has the irreducible submodule $$U({\cA}^{tw}). \Phi(-\lambda + \frac{1}{2}) {\bf 1}. $$
 \end{example}

The following proposition is a twisted version of Proposition 5.1 in \cite{A-2007}.
\begin{proposition} \label{ired-1} Assume that   $p \in \frac{1}{2} {\Zp}$ and that $$\chi (z) = \sum_{k = -2p} ^{\infty} \chi_{-\tfrac{k}{2} } z^{ \tfrac{k}{2} -1}$$ satisfies the following conditions
\bea
&& \chi_p \ne 0, \nonumber \\
&& \chi_0 \in ({\C} \setminus \tfrac{1}{2} {\Z} ) \cup \{ \frac{1}{2} \} \quad \mbox{if} \ p = 0 . \nonumber
\eea
    Then $F^{T_i} (\chi)$ is an irreducible ${\cA}^{tw}$--module with basis:
    $$ G(p-n_1) \cdots G(p-n_r) {\bf 1}, \quad r \ge 0, \ n_i \in \tfrac{1}{2}{\Z}, \  \ n_1 > \cdots > n_r > 0. $$
\end{proposition}
\begin{proof}
We shall prove this proposition in the case $p=0$. The case $p > 0$ uses similar arguments and it is analogous to the aforementioned result in \cite{A-2007}.
Then the action of ${\cA}^{tw}$ is given by
 $$ G(n+1/2) = (\chi_0 -n-1 ) \Phi(n+1/2) +  \sum_{k = 1} ^{\infty} \chi_{-\tfrac{k}{2} } \Phi(n+1/2+ \tfrac{k}{2} ).  $$
 First we shall prove that the vacuum vector is a cyclic vector under
 the $U(\cA^{tw})$--action, i.e.,
 \bea
\label{ciklic} U({\cA}^{tw}). {\1} = F^{T_i}.
 \eea
 Take an arbitrary basis element of $F^{T_i}$:
\bea &&v= \Phi (-n_1) \cdots \Phi(-n_r)
 {\1}   \in F ^{T_i} ,  \label{baza}\eea
 where $n_i \in \tfrac{1}{2} {\Zp}$,  $n_1 >n_2 >\cdots
>n_r \ge \frac{1}{2} $.

 By using action of ${\cA}^{tw}$ we get:
 $$ G(-n_1) G(-n_2)  \cdots G(- n_r) {\1} = C \Phi(-n_1) \Phi(-n_2) \cdots \Phi (-n_r) +  \cdots
\quad (C \ne 0),$$
 where "$\cdots $" denotes the sum of monomials of the form
$$ \Phi(-j_1) \cdots \Phi(-j_l) $$
such that $$ j_1 + \cdots + j_l < n_1 + \cdots + n_r. $$
Using induction on degree of monomials we get
 that $v \in U({\cA}^{tw}). {\1}$.

 For the proof of irreducibility it is enough to see that arbitrary basis element $v$  from (\ref{baza}) is cyclic in $F ^{T_i}$.
 This follows from (\ref{ciklic}) and the fact that
 $$ G(n_1) \cdots G(n_r) v = \nu {\bf 1} \quad (\nu \ne 0),$$
which completes the proof.
\end{proof}

The case $p=0$ and  $\chi_0 \in \tfrac{1}{2} \Z\backslash\{\frac12\}$ is very interesting and as in \cite{A-2013} it will be related to certain polynomials.

\begin{proposition}
Assume that $\chi_0  \in \tfrac{1}{2} {\Z}$, $\chi_0  \le  0$. Then $F^{T_i}(\chi)$ is reducible.
\end{proposition}
\begin{proof}
By using similar arguments to that of  Proposition 4.3 in \cite{A-2013} we see that
$$ \Phi( \chi_0 -1/2). {\bf 1} \notin U({\cA}^{tw})\cdot {\bf 1}, $$
which proves the proposition.
\end{proof}

It remains to consider the case $\chi_0 = \ell +\frac{1}{2}$ for certain  $\ell \in \tfrac{1}{2}{\Z_{\ge 0} } $.
In this case  we have the action of ${\cA}^{tw}$ given by:

\bea &&  G(n+1/2) = (\ell-n-\tfrac{1}{2} ) \Phi(n+1/2) +  \sum_{k = 1} ^{\infty} \chi_{-\tfrac{k}{2} } \Phi(n+1/2+ \tfrac{k}{2} ).  \label{djelovanje} \eea

The following result shows that when  $F^{T_i}(\chi)$ is reducible, its maximal submodule is irreducible.
\begin{proposition} \label{important}
Assume that $0 \ne U \subsetneqq F^{T_i}(\chi)$ is any propoer submodule of  $F^{T_i}(\chi)$. Then there exist uniquely determined $a_{-k/2} \in {\C}$, $k =1,\dots, 2 \ell$ such that
$$ P  _{\ell} = (\Phi(-\ell) + \sum_{k=1}^{2 \ell-1} a_{-k/2} \Phi(-\ell + k/2)  + a_{-\ell} ) {\bf 1} \in U.$$
Moreover, $U^{T_i}(\chi) = U({\cA}^{tw}). P  _{\ell} $ is irreducible and the quotient $F^{T_i} (\chi) / U^{T_i}(\chi)$ is irreducible.
\end{proposition}
\begin{proof}
As in the proof of Proposition \ref{ired-1} we see that
$$ F ^{T_i} (\chi)  = U({\cA}^{tw}). {\bf 1}, $$
which implies that ${\bf 1}$ is a cyclic vector.
Using this and  the  action of operators $G(n)$ one sees that there is a vector $P_{\ell} \in U$ of the form
$$P  _{\ell} = (\Phi(-\ell) + \sum_{k=1}^{2 \ell-1} a_{-k/2} \Phi(-\ell + k/2)  + a_{-\ell} ) {\bf 1}.$$
 Since $U$ is a proper submodule we have
 $$ G(r) P  _{\ell}  = 0 \quad (r > 0), \qquad  ( G(0) + \ell \varepsilon_i) P_{\ell} = 0. $$
 This leads to the following system of linear equations:
 \bea
 (\ell -r) a_{-\ell + r} + \chi_{-1/2} a_{-\ell + r + 1/2} + \cdots + \chi_{-\ell + r+ \tfrac{1}{2}} a_{   - 1/2}   &=& - \chi_{-\ell + r} \quad ( 0 < r < \ell), \nonumber  \\
 2 \ell \varepsilon_i  a_{-\ell} + \chi_{-1/2} a_{-\ell  + 1/2} + \cdots + \chi_{-\ell + 1/2}  a_{-1/2} & = & -\chi_{-\ell} .  \label{sustav}
 \eea
  This linear system (\ref{sustav}) has a unique solution. Irreducibility follows from the fact that the vector $P_{\ell}$ is unique.
\end{proof}

The proof of the following lemma is easy.
\begin{lemma} \label{pmo-3}
  There exist unique $b_{-k/2} \in {\C}$, $k=1, \dots, 4 \ell -1 $ such that
  $$\mathcal{P}_{\ell} := ( G(-\ell) + b_{-1/2} G(-\ell + 1/2) + \cdots+ b_{-\ell}  G(0) + \cdots b_{-2 \ell + 1/2}  G(\ell -1/2) ).  $$
  $$ \{ G(r), \mathcal{P}_{\ell} \} = 0, \quad \mbox{for} \ r=-\ell, \dots, \ell -1 /2.$$
  Moreover,
$$ P_{\ell} = \frac{1}{2 \ell} \mathcal{P}_{\ell}  {\bf 1}$$ and
$$\mathcal{P}_{\ell} ^2 .{\bf 1} = \widetilde{S}_{2 \ell} (\chi_{-1/2}, \cdots, \chi_{- 2 \ell} ) {\bf 1} $$ for certain nonzero polynomial
$\widetilde{S}_{2 \ell}  \in {\C}[x_1, \dots, x_{4 \ell}]$.
\end{lemma}

\begin{theorem}  \label{irr-equiv} The following conditions are equivalent:
\item[(1)] $F ^{T_i} (\chi) $ is irreducible;
\item[(2)] $  \mathcal{P}_{\ell} ^2 .{\bf 1}  \ne 0$;
\item[(3)] The following determinant is nontrivial: \bea \det \left(  \{ G(r), G(s)\} _{ r,s = -\ell, \dots, \ell-1/2}  \right)  \ne 0. \label{cond-irred-2} \eea
\end{theorem}
\begin{proof}
Assume that (2) holds. If $F ^{T_i}(\chi)$ is reducible,  then $U ^{T_i}(\chi)$ is a proper submodule. But,
$$ \mathcal{P}_{\ell} \cdot \mathcal P_{\ell}\cdot {\bf  1} =   \widetilde{S}_{2\ell} (\chi_{-1/2}, \cdots, \chi_{-\ell} ) {\bf 1} \ne 0$$ implies that
${\bf 1} \in U ^{T_i}(\chi)$. This is a contradiction. So $F ^{T_i}(\chi)$ is irreducible.

Assume now that $\mathcal{P}_{\ell} ^2 .{\bf 1} = 0$. This implies that $G(-\ell) P_{\ell}$ is a nontrivial linear combination of
$$ G(-i) P_{\ell}, \quad 0 \le i \le \ell -\frac{1}{2}. $$
This gives that every vector in $U({\cA}^{tw})\cdot P_{\ell}$ is a linear combination of vectors
\bea \label{basis-4}  && G(-n_1) \cdots G(-n_r) P_{\ell}, \quad n_1 > \cdots n_r \ge 0,   \ n_j \ne \ell. \eea
But every vector in (\ref{basis-4}) has non-trivial summand of lowest degree
\bea \label{basis-5}  && \Phi (-n_1) \cdots \Phi(-n_r) \Phi(-\ell) {\bf 1}. \eea
This implies that ${\bf 1} \notin U({\cA}^{tw}) \cdot P_{\ell}$. This proves (1) $\iff$  (2).

Now we shall prove (2) $\iff$ (3).
By construction we have that
$ \{ G(r), \mathcal{P}_{\ell} \} = 0$ for $ r \ge -\ell +1/2$.  Moreover, we have
$$ \{ G(-\ell), \mathcal{P} _{\ell} \} = 0 \iff   \mathcal{P} _{\ell}  ^2 = 0. $$
This implies that ${\mathcal{P}} _{\ell}  ^2 = 0$ if and only if $(1, b_{-1/2}, \cdots, b_{-2\ell + 1/2})$ is the nontrivial solution of the system of linear equations
$$ \{ G(r), x_0 G(-\ell) + \cdots+ x_{-2 \ell + 1/2}  G(\ell -1/2) \} = 0, \quad r=-\ell , \dots, 2 \ell -1/2. $$
The matrix of this linear system is
$$ A_{2 \ell} = \left(  \{ G(r), G(s)\} \right)_{ r,s = -\ell, \dots, \ell-1/2}. $$
Therefore, $\mathcal{P} _{\ell}  ^2 = 0$ if and only if $\det A_{2 \ell} =0$, and this completes the proof.
\end{proof}

\begin{example}
Let $\ell = 1/ 2$. Then:
$$\mathcal{P}_{\ell} = G(-1/2) + \chi_{-1/2} G(0), \quad P_{\ell} = (\Phi(-1/2)  + 2 \chi_{-1/2} \Phi(0) ) {\bf 1}; $$
$$ P_{\ell} ^2  =   \chi_{-1} {\bf 1} .$$
This implies that $F^{T_i}(\chi)$ is irreducible if and only if
$ \chi_{-1} \ne 0. $

Let $\ell =1$. Then
$$\mathcal{P}_{\ell} = G(-1) + 2 \chi_{-1/2} G(-1/2) + (2 \chi_{-1/2} ^2 - 3 \chi_{-1} ) G(0), $$
$$ \mathcal{P}_{\ell} .{\bf 1} = (2 \Phi(-1) + 4 \chi_{-1/2} \Phi(-1/2) + (4 \chi_{-1/2} ^2 - 2 \chi_{-1} ) \Phi(0) ).{\bf 1},  $$
$$ \mathcal{P}_{\ell} ^2  =   8 \chi_{-1} ^2 - 4 \chi_{-2} .$$
This implies that $F^{T_i}(\chi)$ is irreducible if and only if
$ 2 \chi_{-1} ^2 - \chi_{-2}  \ne 0. $

Let $\ell = 3/2 $. By direct calculation we see that
$$ \mathcal{P}_{\ell} ^2 = \nu ( 2 \chi_{-1} ^3 - 3 \chi_{-1} \chi_{-2} +  \chi_{-3}   ) \qquad (\nu \ne 0). $$
Our calculation suggests  that $$\mathcal{P}_{\ell} ^2 = \nu S_{2 \ell} (-2 \chi ^{(1)}  ) = \nu S_{2 \ell} (-2 \chi ^ {(1)}  _{-1} , - 2 \chi  ^ {(1)}  _{-2} ,  \cdots ) . $$
In the following section we shall give a proof of this formula. 
\end{example}

%
%
%

\section{From ${\cV}$--modules to $\widetilde{\cV}$--modules and back}
 \label{korepondencija}
The commutative vertex algebra $M(0)$ generated by fields $\gamma^{\pm} (z)$ can be embedded into the commutative lattice vertex algebra $V_L = M(0) \otimes {\Bbb C} [L]$ associated to the lattice
$$ L= {\Z} \gamma ^{+} + {\Z} \gamma^{-}, \quad \la \gamma ^{\pm}, \gamma ^{\pm} \ra = 0, \quad \la \gamma ^{\pm}, \gamma ^{\mp} \ra = 0. $$
Therefore the vertex superalgebra ${\cV}$ can be embedded into the vertex superalgebra $F \otimes V_L$.

Let $ {\cV} [\Z j]$ be the vertex subalgebra of $F \otimes V_L$ generated by the generators $$ \tau ^{\pm}, j = \frac{\gamma ^+ - \gamma^-}{2}, \nu$$ of ${\cV}$ and by lattice  elements $e ^{j}, e ^{-j}$. Clearly, as a vector space
$${\cV} [\Z j]= {\cV} \otimes C[{\Z} j].$$

Recall that the vertex superalgebra $\widetilde{\cV}$ is generated by $\widetilde{\tau}^{\pm} , \widetilde{\nu}$.  

\begin{proposition} \label{prop-ulaganje}
There exist nontrivial homomorphisms of vertex superalgebras
\bea
\Phi_1 :  \widetilde{\cV} &\rightarrow& {\cV}[\Z j] \label{phi1}\\
\widetilde{\tau} ^{\pm} &\mapsto & \tau^{\pm}_{-1} e ^{\pm j} \nonumber \\
\widetilde{\nu} & \mapsto& \nu + \frac{1}{2}j(-1) ^2 {\bf 1}, \nonumber
\eea
\bea
\Phi_2 :  {\cV} &\rightarrow& \widetilde{\cV} \otimes V_{\Z j} \label{phi2} \\
 {\tau} ^{\pm} &\mapsto &  \widetilde{\tau} ^{\pm}\otimes e ^{\mp j} \nonumber \\
 j &\mapsto& 1 \otimes j \nonumber \\
 {\nu} & \mapsto& \widetilde{\nu }- \frac{1}{2}j(-1) ^2 {\bf 1}. \nonumber
\eea
\end{proposition}
\begin{proof}
By using explicit calculation in lattice vertex algebras we get the following relations:
$$(\tau^{+}_{-1} e ^{ j} )_2  (\tau^{-}_{-1} e ^{ -j} ) = (-2){\bf 1}, \ \  (\tau^{+}_{-1} e ^{ j} )_1  (\tau^{-}_{-1} e ^{ -j} )  =  0,  \ \  (\tau^{+}_{-1} e ^{ j} )_1  (\tau^{-}_{-1} e ^{ -j} ) = 2 (\nu + 1/2 j(-1) ^2).  $$
This proves that  $ \Phi_1$ is a vertex algebra homomorphism. The proof for  the automorphism $\Phi_2$ is similar.
\end{proof}

Assume that $U $ is a   ${\cV}$--module such that \bea \label{cond-1} Y_{U } ( j, z) = \chi^{ (2)} (z) = \frac{  \chi ^{(2)} _{0}} {z} + \sum_{n = 1} ^{\infty} \chi ^{(2)} _{-n} z ^{n-1}, \quad
 \chi ^{(2)} _{0} \in {\Z}.  \eea

Then $U[\Z j] = U\otimes C[{\Z} j] $ has the structure of a  ${\cV}[\Z j]$--module  by letting
$$ Y_{ U[\Z j]  }  (e^{\pm j}, z ) =   z^{  \pm \chi ^{(2)} _{0}   }  E ^- ( \mp \chi^{ (2)} , z) e ^{\pm j}$$
where
$$ E ^{-}  (- \chi^{ (2)}, z) = \exp \left( \sum_{k=1} ^{\infty} \frac{\chi^{ (2)}_{ - k}}{ k}  z ^{ k} \right). $$

Then by using homomorphism $\Phi_1$ we can treat  $U[\Z j]   $ as a  $\widetilde{{\cV}}$--module.   Assume now that $U$ admits a  $\Z$--gradation which is compatible with the
 $\Z $--gradation of  ${\cV}$:
 $$ U = \bigoplus_{s \in \Z} U ^{(s)}, \quad {\cV}^{(i)} . U ^{(s)} \subset U ^{(s+i)} \qquad (s, i \in \Z). $$
 We have its submodule:
 $$\Phi_1(U )  =\bigoplus_{s \in \Z} U ^{(s)} \otimes e ^{s j}.$$

The proof of  the following result uses Proposition \ref{prop-ulaganje} and it  is analogous to the proof of  Theorem 6.2 from \cite{A-2007}.

\begin{proposition} \label{prop-ir-1}
 Assume that $U$ is a ${\cV}$--module which satisfies  the above conditions.
Then $U$ is an irreducible  $ {\cV}$--module if and only if $\Phi_1(U)$ is an irreducible   $\widetilde{\cV}$--module.
\end{proposition}

Assume now that $U^{tw}$ is a twisted  ${\cV}$--module such that \bea \label{cond-2} Y_{U ^{tw} } ( j, z) = \chi^{ (2)} (z) = \sum_{n = 0} ^{\infty} \chi ^{(2)} _{1/2-n} z ^{n-3/2}. \eea
This module has the structure of a twisted  ${\cV}[\Z j]$--module  by letting
$$ Y_{U^{tw} }  (e^{\pm j}, z ) =    E_{tw} ^- ( \mp \chi^{ (2)} , z) $$
where
$$ E_{tw} ^{-}  (- \chi^{ (2)}, z) = \exp \left( \sum_{k=1} ^{\infty} \frac{\chi^{ (2)}_{1/2 - k}}{ k-1/2}  z ^{ k-1/2} \right). $$

Then by using homomorphism $\Phi_1$ we can treat  $U^{tw}  $ as a twisted $\widetilde{{\cV}}$--module, which  we shall denote by $\Phi_1(U^{tw})$.

Let $W^{tw}$ be a twisted $\widetilde{{\cV}}$--module and let $\chi^{ (2)} (z) = \sum_{n = 0} ^{\infty} \chi ^{(2)} _{1/2-n} z ^{n-3/2}$ and $V_{\Z j}(\chi^{(2)})$ the associated $1$--dimensional twisted $V_{\Z j}$--module. Using  homomorphism $\Phi_2$ we can consider $W^{tw} \otimes V_{\Z j} (\chi^{(2)})$ as a twisted   ${\cV}$--module. We shall denote this module by  $\Phi_2(W(\chi^{(2)}))$. Now we have the following result.
\begin{proposition} \label{prop-tw}
For every twisted ${\cV}$--module $U^{tw}$ such that the condition (\ref{cond-2}) holds we have that
$$ \Phi_2 (\Phi_1(U^{tw})(\chi^{(2)}) )\cong U^{tw}. $$
In particular, $U^{tw}$ is irreducible if and only if $\Phi_1(U^{tw})$ is irreducible.
\end{proposition}
\begin{proof}
First we notice that $\Phi_1(U^{tw})= U^{tw}$ as a vector space and the twisted ${\widetilde{\cV}}$ --module action is uniquely defined  by the fields
$$ Y_{\Phi_1(U^{tw} )} (\widetilde{\tau} ^{\pm} ) = Y_{U^{tw}} ({ \tau} ^{\pm}, z )  E_{tw} ^{-}  (\mp \chi^{ (2)}, z) . $$
Let $M = \Phi_2 (\Phi_1(U^{tw})(\chi^{(2)}) )$. Then as a vector space $M= U^{tw}$ and the twisted ${\cV}$--module action is uniquely defined by the fields 
$$Y_{M} (\tau ^{\pm} , z)  = Y_{U^{tw}} ({ \tau} ^{\pm} , z)  E_{tw} ^{-}  (\mp \chi^{ (2)}, z) E_{tw} ^{-}  (\pm \chi^{ (2)}, z)  = Y_{U^{tw}} ({ \tau} ^{\pm} , z) .$$
The proof follows.
\end{proof}

If we apply   Proposition \ref{prop-tw} on the twisted ${\cV}$-- module $F^{T_i} (\chi)$, we shall get the module $\Phi_1 (F^{T_i} (\chi))$ with central character
$$ S(z) = Y^{tw} ( \nu,z) = \frac{1}{2} ( \chi ^{1} (z) ^2 + \partial_z \chi^{(1)}(z) ) = \sum_{n \in \Z} S(n) z ^{-n-2}.$$ So the anti--commutators   $$A(r,s):= \{ G(r), G(s) \} = (-1) ^{2 r +1} ( 2 \delta_{r+s} ^{\Z}    S (r+s)
-  \delta_{r+s} ^{\Z} ( r  ^ 2 - \tfrac{1}{4} ) \ \delta_{r+s,0} )$$ do not depend on $\chi^{(2)}$.
By using the same arguments as in Theorem \ref{irr-equiv} we get

 \bea  \Phi_1 (F^{T_i} (\chi)) \ \mbox{is irreducible} \ \iff \det \left(  \{ A(r,s) \} _{ r,s = -\ell, \dots, \ell-1/2}  \right)  \ne 0. \label{cond-irred--det-2} \eea

By using elementary properties of determinants, one can easily see that
$$  \det \left(  \{ A(r,s) \} _{ r,s = -\ell, \dots, \ell-1/2}  \right)   =  \mu \det(A(\chi)) $$ where $\mu \ne 0$ and
$$ A(\chi) = \left(
\begin{array}
[c]{ccccc}%
2 S(-1)  & 2 S(-2) & \cdots & \  & 2 S(-2 \ell)\\
\ell ^2 - (\ell-1) ^2 & 2 S(-1)  & 2 S(-2) & \cdots & 2 S(-2 \ell -1) \\
0 & \ell ^2 - (\ell -2) ^2 &  2 S(-1) & \cdots & 2 S(-2\ell-2) \\
\vdots & \ddots & \ddots & \ddots & \vdots\\
0 & \cdots & 0 & \ell ^2 - (\ell  -2 \ell +1) ^2 & 2 S(-1)
\end{array} \right) .
$$

Hence we have:

\begin{corollary} \label{ired-expl}
$F^{T_i} (\chi)$ is irreducible if and only if $ \det(A(\chi)) \ne 0$.
In particular,
the irreducibility of $F^{T_i} (\chi)$ depends only on $\chi^{(1)}$.
\end{corollary}

 The irreducibility result for  untwisted ${\cV}$--modules was proved in \cite{A-2013}. But we can also use Proposition \ref{prop-ulaganje} to relate irreducible ${\cV}$--modules and irreducible ${\widetilde{\cV}}$--modules. In the next lemma we shall use this construction to get more precise irreducibility result.

\begin{lemma} \label{det1}
Let
$$S(z) =  \frac{1}{2} ( \chi ^{(1)} (z) ^2 + \partial_z \chi^{(1)} (z) ) = \sum_{n \in \Z} S(n) z ^{-n-2}.$$
Then
$$ \det(A(\chi))   \ne 0 \iff   S_{ 2 \ell } (-2 \chi_{-1} , - 2 \chi_{-2} , \dots  ) \ne 0. $$
\end{lemma}
\begin{proof}
We shall prove the assertion in the case $\ell \in   \tfrac{1}{2} + {\Z}$.  The case $\ell \in  {\Z}$ is similar (see Remark \ref{ramond} below).

Consider now the ${\cV}$--module $\widetilde{F}_{2 \chi ^{(1)} }$ from \cite{A-2013}. By \cite[Theorem 4.1]{A-2013} we have that  $$ \widetilde{F}_{2 \chi ^{(1)} }   \quad \mbox{is irreducible}   \ \iff S_{ 2 \ell } (-2 \chi_{-1} , - 2 \chi_{-2} , \dots  ) \ne 0. $$
If we apply Proposition  \ref{prop-ir-1} for $U= \widetilde{F}_{2 \chi ^{(1)} }$ we shall get the $\widetilde{\cV}$--module  $\Phi_1 (  \widetilde{F}_{2 \chi ^{(1)} } )$ with central character
$$ S(z) = Y( \nu,z) = \frac{1}{2} ( \chi ^{1} (z) ^2 + \partial_z \chi^{(1)}(z) ) = \sum_{n \in \Z} S(n) z ^{-n-2}.$$ So the anti--commutators   are $$A(r,s):= \{ G^+(r), G^-(s) \} =       2 S (r+s)
-  ( r  ^ 2 - \tfrac{1}{4} ) \ \delta_{r+s,0} \quad (r, s \in \tfrac{1}{2} + \Z). $$
By repeating the same arguments as in the proof of  Theorem \ref{irr-equiv} , we get that
$$\Phi_1 (  \widetilde{F}_{2 \chi ^{(1)} } )  \quad \mbox{is irreducible}   \ \iff  \det \left(  \{ A(r,s) \} _{ r,s = -\ell  , \dots, \ell-1}  \right)  \ne 0. $$
Since the last determinant is equal to $ \nu \det(A(\chi))$ for certain $\nu \ne 0$, the proof follows.
\end{proof}

\begin{remark} \label{ramond} When  $\ell \in  {\Z}$, then $\Phi_1 (  \widetilde{F}_{2 \chi ^{(1)} } )$ is a Ramond--twisted module for $\widetilde{\cV}$. All arguments are the same as in the Neveu-Schwarz case above.
\end{remark}

The next theorem  follows from  Lemma \ref{det1}  and Corollary \ref{ired-expl}.

\begin{theorem} \label{conj-tw-1}
$F ^{T_i} (\chi)$ is irreducible if and only if $S_{2 \ell} (-2 \chi ^ {(1)}  _{-1} , - 2 \chi  ^ {(1)}  _{-2} ,  \cdots ) \ne 0$, where
$S_{2 \ell}$ is a Schur polynomial.
\end{theorem}

\section{Twisted $V_{-2}(sl_2)$--modules }
\label{sl2-section}
 Let $V_k(sl_2)$ be the universal affine vertex algebra of level $k$ associated to $\widehat{sl_2}$. Let $e, f, h$ be the standard basis of $sl_2$. For $x \in sl_2$ we identify
 $x$ with $x(-1) {\bf 1}$. Let  $\Theta$ be the automorphism of $V_k(sl_2)$ such that
 $$ \Theta (e ) = f, \ \Theta (f) = e, \ \Theta (h) = -h. $$

 Let $\widehat{sl_2}[\Theta]$ be the affine Lie algebra $\widehat{sl_2}$ in the principal gradation (cf. \cite{FLM}).
Recall that $\widehat{sl_2}[\Theta]$ has basis:
$$\{ K,  \ h(m), \ x^{+} (n), \  x^{-} (p) \ \vert m, p \in \tfrac{1}{2} + \Z, \ n \in {\Z} \}$$
with commutation relations:
\bea
&& [ h(m), h(n) ] =  2 m \delta_{m+n,0} K \nonumber \\
&& [ h(m), x^{+}(r) ] =  2  x ^{-} (m+r) \nonumber \\
&& [ h(m), x^{-}(n) ] =  2  x ^{+} (m+n) \nonumber \\
&& [ x^{+}(r), x^{+}(s) ] =  2  r \delta_{r+s,0} K \nonumber \\
&& [ x^{+} (r), x^{-}(m) ] =  - 2  h  (m+r) \nonumber \\
&& [ x^{-} (m) , x^{-}(n) ] =  - 2 m   \delta_{m+n,0} K \nonumber \\
&& K \ \ \mbox{in the center} \nonumber
\eea
for $m, n \in \tfrac{1}{2} + \Z$, $r, s \in \Z$.
Define:
\begin{align*}X_{\pm} (z) &= \frac{1}{2} ( \sum_{n \in {\Z}} x^{+} (n) z ^{-n-1} \pm \sum_{n \in  \tfrac{1}{2} + {\Z} } x^{-} (n)  z^{-n-1} ), \\
h(z) &= \sum_{ n \in \tfrac{1}{2} + {\Z} } h(n) z ^{-n-1}.
\end{align*}

Recall   (cf. \cite{HJM})  that  the irreducible highest weight $\widehat{sl_2}[\Theta]$--module $L( \frac{k}{2} (\Lambda_0 + \Lambda_1) + j (\Lambda_1-\Lambda_0) )$ of level $k$ is  generated by highest weight vector $v_j$ which satisfies condition
$$ x^{+} (n) v_j = x^- (p) v_j = h(m) v_j = 0 \quad (n, m , p > 0), \quad x^+ (0) v_j = j v_j. $$

 The following result is standard (see \cite{FLM}).
\begin{proposition}
The category of $\Theta$--twisted $V_{k }(sl_2)$--modules coincides with the category of restricted modules for $\widehat{sl_2}[\Theta]$ of level $k$.
\end{proposition}

Recall that (cf. \cite{A-2007}) the universal affine vertex algebra $V_{-2} (sl_2)$ can be realized as a subalgebra of ${\cV} \otimes F_{-1}$  with generators
$$ e = \tau ^+ \otimes e ^{\beta}, \ h = - 2 \beta(-1) + 2  j, \ f = \tau ^- \otimes e ^{-\beta}. $$

 By Section 6 of  \cite{A-2007} we have that $\mathcal{\cV} \otimes F_{-1}$ is  ${\Z}$--graded:
$$\mathcal{\cV} \otimes F_{-1} = \oplus_{s \in {\Z} } W(s)$$
where $W(0) $ is a subalgebra of $\mathcal{\cV} \otimes F_{-1}$ generated by $e,f,h $ and $j = T(-1) {\bf 1} \otimes {\bf 1}$, and $W(s) = W(0). ({\bf 1} \otimes  e ^{s \beta}). $ Moreover,
$$ W(0) \cong V_{-2} (sl_2) \otimes M_T(0), \quad W(s) \cong \pi_s( V_{-2} (sl_2) ) \otimes M_T(0), $$
where $$\pi_s(V_{-2} (sl_2) ) = V_{-2} (sl_2) . e^ {s \beta}. $$ Since $M_T(0) \subset {\cV}$ we get:

\begin{proposition} \label{description}
$$\widetilde{\cV} \otimes F_{-1} \cong \bigoplus _{s \in {\Z} } \pi_s (V_{-2} (sl_2)). $$
In particular,
$$V_{-2} (sl_2) \cong \bigoplus_{i \in \Z}  \widetilde{\cV} ^{(i)} \otimes F_{-1} ^{(i)}. $$
\end{proposition}

Let $\Theta = \Theta_{\cV} \otimes \Theta_{F_{-1}}$. Then $$ \Theta (e) = f, \ \Theta (f) = e, \ \Theta (h) = -h. $$
Therefore $\Theta$ can be considered as an automorphism of the vertex algebra $V_{-2}(sl_2)$.

The following result is analogous to the main result of \cite{A-2007}.

\begin{theorem} \label{main-t1}
Assume that $U^{tw}$ is any irreducible $\Theta_{\cV}$--twisted ${\cV}$--module. Then for $i=1,2$
$$  U ^{tw} \otimes F_{-1} ^{T_i}$$
is an irreducible $\Theta$--twisted $V_{-2}(sl_2)$--module. In particular, $  U ^{tw} \otimes F_{-1} ^{T_i}$ is an irreducible module for $\widehat{sl_2}[\Theta]$ at the critical level. The action of  $\widehat{sl_2}[\Theta]$ is uniquely determined by the components of the fields:
$$ X_{+} (z) =  \frac{1}{\sqrt{2} } \mathcal{G} (z) \otimes Y_{tw} (e ^{\pm\beta}, z), \ h(z) =-2 \beta (z) + 2 \mathcal{T}(z). $$
\end{theorem}
\begin{proof}
First we notice that $ M^{tw}:= U ^{tw} \otimes F_{-1} ^{T_i}$ is an irreducible $\Theta$--twisted module for the vertex superalgebra $\mathcal{\cV} \otimes F_{-1}$. Let $Y_{M ^{tw}}(\cdot,z)$ be the associated vertex operator. For $ u \in \mathcal{\cV} \otimes F_{-1}$, let
$$ Y ^{M ^{tw}} (u, z) = \sum_{n \in {\hf} {\Z} } u ^{tw} _n z ^{-n-1}. $$

  Using the twisted version of Corollary 4.2 in \cite{DM-galois} we get that for each $w \in M^{tw}$,
$$M ^{tw} = \mathcal{\cV} \otimes F_{-1} . w =  \mbox{span}_{\C} \{ u^{tw}_ n w \ \vert \ u  \in  \mathcal{\cV} \otimes F_{-1}, \ n \in \hf {\Z}  \}. $$
By using explicit vertex operator formulas  (\ref{for-vo1})-(\ref{for-vo2}) we get that
$$Y_{M ^{tw}}({\bf 1} \otimes e ^{s \beta} ,z) w \in  W(0). w [[ z^{1/2}, z ^{-1/2}]]. $$
This implies that $$ M ^{tw} = W(0).w = \mbox{span}_{\C} \{ u^{tw}_ n w \ \vert \ u  \in  W(0), \ n \in \hf {\Z}. \}. $$
Therefore $M ^{tw}$ is an irreducible twisted $W(0)$--module. Since $W(0) \cong V_{-2}(sl_2) \otimes M_T(0)$--module, where $M_T(0)$ is a commutative subalgebra of $W(0)$ generated by $T(n)$, $n \in {\Z}$ we conclude that   $M^{tw}$ is an irreducible $\Theta$--twisted   $V_{-2}(sl_2)$--module which completes the proof.
\end{proof}

\begin{theorem}\label{e:twisted}
Assume that $\chi \in {\C} (( z ^{1/2} ))$, $i,j \in \{1,2\}$. Then
$$ F ^{T_i} (\chi) \otimes F_{-1} ^{T_j} $$
is a $\Theta$--twisted $V_{-2}(sl_2)$--module.

In particular, if ${\lambda} \in \{ 1 /2 \} \cup  ({\C} \setminus \tfrac{1}{2} {\Z})$, then  $F^{T_i} (\frac{\lambda}{z}) \otimes F^{T_j}$ is an irreducible $\widehat{sl_2}[\Theta]$--module at the critical level isomorphic to
$$ L( (  2 \lambda-2) \Lambda_0  -2 \lambda \Lambda_1 ) \quad \mbox{if} \ \quad i =j , $$
 $$ L( -2 \lambda  \Lambda_0  +  (2 \lambda -2 )  \Lambda_1 ) \quad \mbox{if} \ \quad i \ne j . $$
 Hence
 $$F^{T_i} (\frac{1}{2 z}) \otimes F^{T_j} \cong L(-\rho) =L(-\Lambda_0 - \Lambda_1) \quad \forall i, j. $$
\end{theorem}

By using the boson-fermion correspondence we have that $F^{T_i}$ can be realized on $M_{\hf + \Z } (1)$. Therefore, we get the following result.

\begin{corollary} For every $\chi \in {\C} ((z ^{1/2}))$ we have the   $\widehat{sl_2}[\Theta]$--module structure on the vector space
$$M_{\hf + \Z } (1) \otimes M_{\hf + \Z } (1). $$
\end{corollary}

Assume now that $\lambda \in \tfrac{1}{2} \Z$, $\lambda < 1/2$. Then the action of ${\cA} ^{tw}$ on $F^{T_i} ( \lambda /z)$ is uniquely determined by
$$ G ({n+ 1/2}) = - (n+1- \lambda)  \Phi (n+1/2), \quad (n \in \tfrac{1}{2} {\Z} ). $$ Recall  that $F^{T_i} ( \lambda /z)$ has an irreducible submodule:
$$ \overline{F} ^{T_i} (\lambda / z) = \mbox{Ker}_{ \overline{F} ^{T_i}  }  \Phi (1/2- \lambda )   \cong \bigwedge
\left(   \Phi(n)  \ \vert \ n \in \tfrac{1}{2} {\Z}, n < 0,  n \ne  \lambda -1/2 \right). $$

\begin{theorem} Assume that  $\lambda \in \tfrac{1}{2} \Z$, $\lambda < 1/2$.
Then
$$ \overline{F} ^{T_i} (\lambda /z)  \otimes F^{T_j} = L(( 2 \lambda-2) \Lambda_0 - 2 \lambda \Lambda_1)  \quad \mbox{if} \quad i=j , $$
$$ \overline{F} ^{T_i} (\lambda / z) \otimes F^{T_j} = L(-2 \lambda \Lambda_0 + ( 2 \lambda -2 ) \Lambda_1) \quad \mbox{if} \quad i\ne j . $$

\end{theorem}

\subsection{Example:}
In the case of untwisted representations,  $\chi = 0$ gave  the realization of the vacuum module $L(-2 \Lambda_0)$  \cite{A-2007}.
Now we shall consider the twisted case when  $\chi = 0$. Let us consider $F^{T_i} (0) = F ^{T_i} $ as a module over Lie superalgebra ${\cA} ^{tw}$. Then the action of the generators of ${\cA} ^{tw}$ is given by
$$ G (n+1/2) = - (n+1)  \Phi (n+1/2), \quad (n \in \tfrac{1}{2} {\Z} ), $$
and other generators act trivially. So, $F ^{T_i} $ has a simple ${\cA} ^{tw}$--submodule:
$$ \overline{F} ^{T_i} = \mbox{Ker}_{ \overline{F} ^{T_i}  }  \Phi (1/2)   \cong \bigwedge \left(   \Phi(n)  \ \vert \ n \in \tfrac{1}{2} {\Z}, n < 0,  n \ne -1/2 \right). $$

We have:
$$ \overline{F} ^{T_i} \otimes F^{T_j} = L(-2 \Lambda_0) \quad \mbox{if} \quad i=j , $$
$$ \overline{F} ^{T_i} \otimes F^{T_j} = L(-2 \Lambda_1) \quad \mbox{if} \quad i\ne j . $$

\subsection{Characters for highest weight modules at the critical level}

In the previous section we presented an explicit realization of all irreducible highest weight $\widehat{sl_2}[\Theta]$--modules (in the principal picture).
Let us now determine its $q$--characters.

By construction, ${F} ^{T_i} (\lambda /z)   \otimes F^{T_j} $ are realized on the  $\Theta_{F} \otimes \Theta_{F_{-1}}$--twisted module  ${F} ^{T_i} \otimes F^{T_j}$. Recall that $F \otimes F_{-1}$ has the following Virasoro vector (cf. \cite{A-2007}):
$$\omega = \omega ^{(f)} - 1/2 \beta(-1) ^{2}, \quad \omega^{(f)} = \frac{1}{2} (\Psi ^{+} (-3/2) \Psi ^{-} (-1/2) + \Psi^{-}(-3/2) \Psi^{+}(-1) ) {\bf 1}. $$
Let $ L^{tw} (z) = Y_{tw} (\omega, z) = \sum_{n \in {\Z} } L^{tw} (n) z ^{-n-1}$. The operators $L^{tw}(n), n \in {\Z} $ define on ${F} ^{T_i} \otimes F^{T_j}$ the representation of the Virasoro algebra with central charge $2$.  One can see that $\widehat{sl_2}[\Theta]$ action on $F^{T_i} (\chi) \otimes F_{-1} ^{T_j}$ is compatible with the $L ^{tw} (0)$--gradation if and only if $\chi (z) = \lambda /z$. Then  $F^{T_i} (\chi) \otimes F_{-1} ^{T_j}$ becomes a $\widehat{sl_2}[\Theta] \oplus L^{tw}(0)$--module.

We have the following $q$--character:
$$\mbox{ch}_ { {F} ^{T_i} (\lambda /z)   \otimes F^{T_j}} (q) =tr q^{L ^{tw} (0)} \sim  \prod_{n=1 } ^{\infty} (1 - q^{n-1/2}) ^{-2},$$

which is in complete agreement with the Kac-Kazhdan character formula.

%
%

\section{Relations with $Z$--algebras}

In this section we shall study connection with the theory of $Z$--algebras and extend the results from \cite{DJM}. As a consequence we shall construct combinatorial  bases of  parafermion type for  a family of irreducible modules at the critical level.

Let $V$ be any quotient of the vertex algebra $\cV$.
Consider now the vacuum subspace of $V$:
$$ \Omega (V) = \{ v \in V_{-2}( sl_2) \ \vert \ h(n) v = 0 \ n \ge 1 \}. $$
It was proved in \cite{GL} that there is a unique generalized vertex algebra structure $Y_{\Omega}$  on $\Omega ( V)$ generated by the fields

$$ A^+ (z)= Y_{\Omega} (e(-1) {\1}, z) = E^{-} (-\tfrac{1}{2}  h,z  ) Y( e (-1) {\1}, z ) E ^{+} ( -\tfrac{1}{2} h, z ) z ^{\tfrac{1}{2} h(0)}, $$
$$ A^- (z) =Y_{\Omega} (f(-1) {\1}, z) = E^{-} (\tfrac{1}{2}  h, z ) Y( f (-1) {\1}, z ) E ^{+} (  \tfrac{1}{2} h, z ) z ^{-\tfrac{1}{2} h(0)}. $$
where as usual we set for $\alpha \in {\C} h $
$$E ^{\pm} ( \alpha, z) = \exp \left( \sum_{ \pm n \in {\N} } \frac{\alpha (n) }{n} z ^{-n} \right).  $$

\begin{theorem} \label{ident-vacuum} We have:
\item[(1)] $ \Omega (V_{-2} (sl_2) ) \cong \widetilde{\mathcal{V}}$
\item[(2)] $ \Omega(L_{-2} (sl_2) ) \cong \overline{F}$.
\end{theorem}
\begin{proof}
By explicit realization and description of $V_{-2}(sl_2)$ from Proposition \ref{description} we easily see that
 $$ \Omega (V_{-2} (sl_2) ) \cong \bigoplus_{i \in \Z} \widetilde{\mathcal{V}} ^{(i)} \otimes e ^{i \beta} $$
 which implies that
 $ \Omega (V_{-2} (sl_2) ) \cong \widetilde{\mathcal{V}}$ as vector spaces.
By \cite{GL} we know that $ \Omega (V_{-2} (sl_2) ) $ is generated by
$A^+(z)$ and $A^-(z)$. We have:
$$ A^+ (z) =  G^+ (z)  E^{-} (  \beta, z )  E ^- (- \beta, z )   E ^+ ( \beta, z) E ^{+} ( -\beta, z  ) z ^{ \beta (0) } z ^{-\beta (0) } = G ^+ (z) $$
and similarly $A ^- (z) = G  ^- (z)$. This proves assertion (1). The proof of assertion (2)
uses similar arguments and description of the simple vertex algebra $L_{-2} (sl_2)$  from
Corollary 8.1. of \cite{A-2007}.
\end{proof}

\begin{remark}
In \cite{GL}, the authors defined  certain  category $\mathcal{C}_k$ whose objects  are $V_k(sl_2)$--modules $W$ which are weight $\widehat{{\frak h}}$--modules such that  $\widehat {\frak h}^+$ acts locally nilpotently. They also showed  that for every (irreducible) $ \Omega (V_k(sl_2) ) $--module $U$ on the vector space $U \otimes M_{\Z}(1)$ there is a structure of (irreducible) $V_k(sl_2)$--module from the category  $\mathcal{C}_k$ (for details see \cite{GL}; see also \cite{DL}).


We should also say that there exists interesting $V_k(sl_2)$--modules which don't belong to the category $\mathcal{C}_k$ and which are not of the form $U \otimes M_{\Z}(1)$. In particular, Whittaker modules  can not be obtained using the above methods.
\end{remark}

Let $\mathcal{C}_k ^{tw} $ be category of $\Theta$--twisted $V_{k}(sl_2)$--modules $W$  on which $\hat{\h}^+ _{\Z + \frac{1}{2} }$ acts   locally nilpotently.

The generalized vertex algebra $ \Omega(V_k(sl_2))$ admits a canonical automorphism of order two induced from the automorphism $\Theta$ of $V_k(sl_2)$ (for discussion on automorphisms of $ \Omega(V_k(sl_2))$ see \cite{Li-abelian}). Then by using analogous proof to that of \cite{GL} and \cite{Li-abelian} one can prove that for every $\Theta$--twisted $ \Omega(V_k(sl_2))$--module $U ^{tw}$ on the vector space $U ^{tw} \otimes M_{\tfrac{1}{2} + {\Z} } (1)$ there is a structure of  (irreducible) $\Theta$--twisted $V_k(sl_2)$--modules. In particular, $U ^{tw} \otimes M_{\tfrac{1}{2} + {\Z} } (1)$ is a $\widehat{sl_2}[\Theta]$--module. In the case of critical level, since  $\Omega (V_{-2} (sl_2) ) \cong \widetilde{\cV}$, our construction proves that  $U ^{tw} \otimes M_{\tfrac{1}{2} + {\Z} } (1)$ can be equipped with the structure  of an  (irreducible) $\Theta$--twisted $V_k(sl_2)$--module in the category $\mathcal{C}_k ^{tw} $ for $k=-2$.  On the other hand for every $\Theta$--twisted $V_k(sl_2)$--module $W$ from the category $\mathcal{C}_k ^{tw} $, let
$$\Omega(W) = \{ v \in W \ \vert  \ h (n) v = 0 , \quad n \in \tfrac{1}{2} + {\Z}_+ \}. $$

\begin{theorem} Assume that $(W,Y_W ^{tw})$ is any $\Theta$--twisted $V_{-2}(sl_2)$-module from the category $\mathcal{C}_{-2} ^{tw}$. Then $( \Omega(W), Y_{\Omega} ^{tw})$ has the structure of a 
$\Theta_{\widetilde {\cV}}$--twisted  $\widetilde{\mathcal{{\cV} }}$--module uniquely determined by the vertex operators
 \begin{align*}
 A^+ (z) &= Y_{\Omega} ^{tw} (e(-1) {\bf 1}, z) = E^{-} _{tw} (-\frac{h}2, z)  Y ^{tw} (e(-1){\bf 1},z)  E^{+} _{tw} (-\frac{h}2, z),\\
A^- (z) &= Y_{\Omega} ^{tw} (f(-1) {\bf 1}, z)  = E^{-} _{tw} (\frac{h}2, z)  Y ^{tw} (f(-1){\bf 1},z)  E^{+} _{tw} (\frac{h}2, z).
\end{align*}
Assume that  $W$ is isomorphic to a module $ U^{tw} \otimes F_{-1}^{T_j}$ constructed from the irreducible ${\cA}^{tw}$--module $U^{tw}$  such that for $n >0$  $\mathcal T  (n) = 0$ on $ U^{tw}$ (see Section \ref{constr-2}). Then $\Omega (W) \cong U^{tw}$ as a vector space and the action of $A^{\pm}(z)$ is as follows:
$$ A^+(z) = G^+(z) E_{tw} ^-  (-j,z )= \sum_{n \in \tfrac{1}{2} \Z} \widetilde{G}(n) z ^{-n-1},$$
and $A^-(z)|_{z^{1/2}\to -z^{1/2}}=A^+(z)$.
\end{theorem}
\begin{proof} Note that $Y^{tw}(e(-1){\bf 1}, z)=X_{+}(z), Y^{tw}(f(-1){\bf 1}, z)=X_{-}(z)$ and
recall that the OPE's of the twisted affine Lie algebra \cite{FLM} are given by
\begin{align*}
X_{+}(z)X_{-}(w)\sim \frac{\alpha(w)}{2(z^{1/2}-w^{1/2})}+\frac{C}{4(z^{1/2}-w^{1/2})^2},
\end{align*}
where $\alpha(z)=\sum_{n\in\mathbb Z_++1/2}\alpha(n)z^{-n-1/2}$ and $X_{\pm}(z)=\sum_{n\in\mathbb Z/2}X_{\pm}(n)z^{-n-1/2}$.

On the other hand we also have
\begin{align*}
&E^{+} _{tw} (\alpha, z)E^{-} _{tw} (\beta, w)=E^{-} _{tw} (\beta, w)E^{+} _{tw} (\alpha, z)
(\frac{z^{1/2}-w^{1/2}}{z^{1/2}+w^{1/2}})^{(\alpha, \beta)}\\
&E^{\pm}(\alpha, z)=E^{\pm}(\alpha, w)\left( 1-2\alpha_{\mp}(w)(z^{1/2}-w^{1/2}) \right. \\
&\hskip 1in \left. -2\alpha'_{\mp}(w)(z^{1/2}-w^{1/2})^2+\cdots\right)
\end{align*}
where $\alpha(z)_{\pm}=\sum_{\pm n\in\mathbb N-1/2}\alpha(-n)z^{n-1/2}$
and $\alpha'(z)_{\pm}=\partial_z\alpha_{\pm}(z)$.
Using this we obtain that
\begin{align*}
&A^+(z)A^-(w)=E^{-} _{tw} (-\frac{h}2, z)X_+(z)  E^{+} _{tw} (-\frac{h}2, z)
E^{-} _{tw} (\frac{h}2, w)  X_{-}(w)  E^{+} _{tw} (\frac{h}2, w)\\
&=E^{-} _{tw} (-\frac{h}2, z)E^{-} _{tw} (\frac{h}2, w)X_+(z)X_{-}(w)  E^{+} _{tw} (-\frac{h}2, z)
    E^{+} _{tw} (\frac{h}2, w)
(\frac{z^{1/2}+w^{1/2}}{z^{1/2}-w^{1/2}})\\
&\sim \frac{2w^{1/2}:X_+(w)X_{-}(w):}{z^{1/2}-w^{1/2}}+
\frac{E^{-} _{tw} (-\frac{h}2, z)E^{-} _{tw} (\frac{h}2, w)h(w)E^{+} _{tw}(-\frac{h}2, z)E^{+} _{tw} (\frac{h}2, w)} {2(z^{1/2}-w^{1/2})^2(z^{1/2}+w^{1/2})^{-1}}\\
&\hskip 1in -\frac{E^{-} _{tw} (-\frac{h}2, z)E^{-} _{tw} (\frac{h}2, w)E^{+} _{tw}(-\frac{h}2, z)E^{+} _{tw} (\frac{h}2, w)} {2(z^{1/2}-w^{1/2})^3(z^{1/2}+w^{1/2})^{-1}}
\end{align*}
which is simplified to the following:
\begin{align*}
&\sim \frac{2w^{1/2}:X_+(w)X_{-\alpha}(w):}{z^{1/2}-w^{1/2}}+\frac{w^{1/2}(:h(w)^2:-h'(w))}{z^{1/2}-w^{1/2}}
-\frac{w^{1/2}}{(z^{1/2}-w^{1/2})^3}.
\end{align*}

Note that the operator $$ S(z)= \frac{1}{2}  (2  :X_{\alpha}(z)X_{-\alpha}(z): +  :h(z)^2:-h'(z) )$$ 
is central at the critical level  and it 
represents $\mathcal {S}(z)$. This shows that the operators $A^+(z)$ and $A^-(w)$ satisfy the twisted superconformal
algebra defined by $\mathcal G^{\pm}(z)$, where $\mathcal{T}(z)$ acts trivially.
This proves that  $( \Omega(W), Y_{\Omega} ^{tw})$ is a $\Theta_{\widetilde {\cV}}$--twisted  $\widetilde{\mathcal{{\cV} }}$--module.

\end{proof}

In view of parametrization of the twisted modules by $\chi$ in Theorem \ref{e:twisted}, the following
description of the general vacuum space is clear from the above discussion.
 \begin{proposition} \label{ired-expl-2}
 Assume that $ \chi(z) = \chi^{(1)} (z) + \chi^{(2)} (z)$ where
 $$\chi^{(1)} (z) = \sum_{n=0} ^{\infty} \chi_{-n} z ^{n-1}, \  \chi^{(2)} (z) = \sum_{n=0} ^{\infty} \chi_{-n-1/2} z ^{n-1/2} . $$
 Then
$$\Omega ( F^{T_i} (\chi) \otimes F_{-1} ^{T_j} ) = \Phi_1 ( F ^{T_i} (\chi ) ) $$
 as $\Theta$--twisted $\widetilde{\cV}$--modules. In particular, if $ \chi^{(2)} (z) = 0$, then  $F ^{T_i} (\chi ) $   is a  $\Theta_{\widetilde {\cV}}$--twisted $\widetilde{\cV}$--module and $\Omega ( F^{T_i} (\chi) \otimes F_{-1} ^{T_j} )  \cong F ^{T_i} (\chi ) $.
 \end{proposition}

\subsection{Combinatorial bases of modules at the critical level}

Our approach can be used to construct a basis for the twisted module at the critical level. 

\begin{corollary} \label{bases-1}Assume that
$\chi (z) \in {\C}(z^{\tfrac{1}{2}} )$ such that one of the conditions (1.1)-(1.3) of Theorem \ref{ired-uvod} hold. Then
$F ^{T_i}( \chi)$ is a irreducible ${\cA}^{tw}$--module and   the set of vectors
$$  G(p-n_1) G( p-n_2) \cdots G(p-n_r) h(-m_1) \cdots h(-m_s) ({\bf 1} \otimes {\bf 1} ) $$
such that
$ r, s \ge 0$, $n_i, m_j \in \tfrac{1}{2}{\Zp}$ and
$$ n_1 > n_2 > \cdots >n_r > 0 , \quad m_1 \ge m_2 \ge \cdots \ge  m_s > 0$$
is a basis of the irreducible $\Theta$--twisted $V_{-2}(sl_2)$--module $F^{T_i} (\chi) \otimes F_{-1} ^{T_j}. $
\end{corollary}

 Corollary \ref{bases-1} also  describes a combinatorial basis of irreducible highest weight module $L((-2-t) \Lambda_0 + t \Lambda_1) $ where  $t \in \{-1\} \cup  ( {\C} \setminus {\Z} )$. The bases of remaining irreducible highest weight modules are described in the following result. 

\begin{corollary} Assume that $\mu_1 = (-2 -t ) \Lambda_0 + t \Lambda_1$, $\mu_2 = t \Lambda_0 - (2 + t) \Lambda_1 $ for $t \in {\Zp}$.
Then the set of vectors
$$  G(-n_1) G(-n_2) \cdots G(-n_r) h(-m_1) \cdots h(-m_s) v_{\mu_i}  $$
such that
$ r, s \ge 0$, $n_i, m_j \in \tfrac{1}{2}{\Zp}$, $n_i \ne \tfrac{t}{2} + \tfrac{1}{2}$,  and
$$ n_1 > n_2 > \cdots >n_r > 0 , \quad m_1 \ge m_2 \ge \cdots \ge  m_s > 0$$
is a basis of $L(\mu_i)$, $i=1,2$ in the principal picture.
\end{corollary}

These two results   give the proof of Theorem \ref{th1.1-intr} from   the Introduction.
\bigskip

\centerline{\bf ACKNOWLEDGMENTS}
Part of the research was carried out during the visit of D.A. to North Carolina State University. 
D.A. is partially supported by the Croatian Science Foundation under the project 2634.
NJ would like to thank the support of
Simons Foundation 
and National Natural Science Foundation of China. 
KCM would like to thank the support of Simons Foundation.

\end{document}